\documentclass[oneside,english]{amsart}
\usepackage[T1]{fontenc}
\usepackage[latin9]{inputenc}
\usepackage{babel}
\usepackage{array}
\usepackage{float}
\usepackage{booktabs}
\usepackage{url}
\usepackage{multirow}
\usepackage{amstext}
\usepackage{amsthm}
\usepackage{amssymb}
\usepackage{graphicx}
\usepackage{esint}
\usepackage[unicode=true,pdfusetitle,
 bookmarks=true,bookmarksnumbered=false,bookmarksopen=false,
 breaklinks=false,pdfborder={0 0 1},backref=page,colorlinks=false]
 {hyperref}

\makeatletter

\providecommand{\tabularnewline}{\\}
\newcommand{\lyxdot}{.}

\numberwithin{equation}{section}
\theoremstyle{plain}
\newtheorem{thm}{\protect\theoremname}
  \theoremstyle{definition}
  \newtheorem{defn}[thm]{\protect\definitionname}
  \theoremstyle{plain}
  \newtheorem{prop}[thm]{\protect\propositionname}
  \theoremstyle{plain}
  \newtheorem{lem}[thm]{\protect\lemmaname}
  \theoremstyle{plain}
  \newtheorem{conjecture}[thm]{\protect\conjecturename}
  \theoremstyle{plain}
  \newtheorem{question}[thm]{\protect\questionname}
  \theoremstyle{remark}
  \newtheorem{rem}[thm]{\protect\remarkname}
  \theoremstyle{plain}
  \newtheorem{cor}[thm]{\protect\corollaryname}
\newcommand{\strong}[1]{\textbf{#1}}

\usepackage[labelfont=normal]{subfig}


\@ifundefined{showcaptionsetup}{}{%
 \PassOptionsToPackage{caption=false}{subfig}}
\usepackage{subfig}
\makeatother

  \providecommand{\conjecturename}{Conjecture}
  \providecommand{\corollaryname}{Corollary}
  \providecommand{\definitionname}{Definition}
  \providecommand{\lemmaname}{Lemma}
  \providecommand{\propositionname}{Proposition}
  \providecommand{\questionname}{Question}
  \providecommand{\remarkname}{Remark}
\providecommand{\theoremname}{Theorem}

\begin{document}
\global\long\def\colonsubseteq{\colon{\subseteq{}}}

\title{When does randomness come from randomness?}

\author{Jason Rute}

\address{Department of Mathematics\\
Pennsylvania State University\\
University Park, PA 16802 }

\email{jmr71@math.psu.edu}

\urladdr{\url{http://www.personal.psu.edu/jmr71/}}


\thanks{This work was started while the author was participating in the
Program on Algorithmic Randomness at the Institute for Mathematical
Sciences of the National University of Singapore in June 2014. The
author would like to thank the institute for its support.}

\maketitle

\begin{abstract}
A result of Shen says that if $F\colon2^{\mathbb{N}}\rightarrow2^{\mathbb{N}}$
is an almost-everywhere computable, measure-preserving transformation,
and $y\in2^{\mathbb{N}}$ is Martin-Löf random, then there is a Martin-Löf
random $x\in2^{\mathbb{N}}$ such that $F(x)=y$. Answering a question
of Bienvenu and Porter, we show that this property holds for computable
randomness, but not Schnorr randomness. These results, combined with
other known results, imply that the set of Martin-Löf randoms is the
largest subset of $2^{\mathbb{N}}$ satisfying this property and also
satisfying randomness conservation: if $F\colon2^{\mathbb{N}}\rightarrow2^{\mathbb{N}}$
is an almost-everywhere computable, measure-preserving map, and if
$x\in2^{\mathbb{N}}$ is random, then $F(x)$ is random.
\end{abstract}

\section{Introduction}

Algorithmic randomness is a branch of mathematics which gives a rigorous
formulation of randomness using computability theory. The first algorithmic
randomness notion, Martin-Löf randomness, was formulated by Martin-Löf
\cite{Martin-Lof:1966mz} and has remained the dominant notion in
the literature. Schnorr \cite{Schnorr1971}, however, felt that Martin-Löf
randomness was too strong, and introduced two weaker randomness notions
now known as Schnorr randomness and computable randomness.

While, historically randomness has mostly been studied on Cantor space
$2^{\mathbb{N}}$ with the fair-coin measure $\lambda$, there has
been a lot of interest lately in the behavior of algorithmic randomness
on other spaces and measures. Many of these results are of the form,
``A point $y\in\mathbb{Y}$ is $\nu$-random (where $\nu$ is a Borel
probability measure on $\mathbb{Y}$) if and only if $y=F(x)$ for
some $\mu$-random $x\in\mathbb{X}$ (where $\mu$ is a Borel probability
measure on $\mathbb{X}$).''

As an example, consider von Neumann's coin. Von Neumann showed that
given a possibly biased coin with weight $p\in(0,1)$, one can recover
the distribution of a fair coin by following this given procedure:
Toss the coin twice. If the results match, start over, forgetting
both results. If the results differ, use the first result, forgetting
the second. Von Neumann's procedure describes a partial computable
function $F\colonsubseteq2^{\mathbb{N}}\rightarrow2^{\mathbb{N}}$
whereby an infinite sequence of independent and identically distributed
biased coin tosses $x\in2^{\mathbb{N}}$ is turned into an infinite
sequence of independent and identically distributed fair coin tosses
$F(x)$.

Now, as for randomness, one can prove that for a fixed computable
$p\in(0,1)$, a sequence $y\in2^{\mathbb{N}}$ is Martin-Löf random
for the fair-coin measure if and only if $y$ can be constructed via
the von Neumann procedure starting with some $x\in2^{\mathbb{N}}$
which is random for the $p$-Bernoulli measure ($p$-weighted coin
measure). While there are many methods available in algorithmic randomness
to prove this, the easiest is to just apply the following theorem. 
\begin{thm}[{See \cite[Thms.~3.2, 3.5]{BienvenuSubmitted}\cite[Prop.~5]{Hoyrup:2009pi}.}]
\label{thm:ML-rp-nrfn}Assume $\mu$ and $\nu$ are computable probability
measures on $2^{\mathbb{N}}$ and that the map $F\colon(2^{\mathbb{N}},\mu)\rightarrow(2^{\mathbb{N}},\nu)$
is almost-everywhere computable ($F$ is computable on a $\mu$-measure-one
set) and measure-preserving ($\nu(B)=\mu(F^{-1}(B))$ for all Borel
$B$).
\begin{enumerate}
\item If $x$ is $\mu$-Martin-Löf random then $F(x)$ is $\nu$-Martin-Löf
random.
\item If $y$ is $\nu$-Martin-Löf random then $y=F(x)$ for some $\mu$-Martin-Löf
random $x$.
\end{enumerate}
\end{thm}
The first half of Theorem~\ref{thm:ML-rp-nrfn} is known as \emph{randomness conservation},
\emph{randomness preservation}\footnote{Simpson and Stephan \cite{Simpson:aa} use the term ``randomness
preservation'' for another property: if $x$ is Martin-Löf random,
then there is a PA degree $p$ such that $x$ is Martin-Löf random
relative to $p$.},  or \emph{conservation of randomness}. This result, at least in
the finitary setting of Kolmogorov complexity, goes back to Levin
\cite[Thm.~1]{Levin:1984}. (See Gács \cite{Gacs:aa}.) The second
half is known as \emph{no-randomness-from-nothing} or \emph{no randomness ex nihilo}.
Bienvenu and Porter \cite{BienvenuSubmitted} attribute it as an unpublished
result of Alexander Shen.\footnote{There is an error in the proof of no-randomness-from-nothing in \cite[Thm.~3.5]{BienvenuSubmitted}.
The authors say ``Since $\Phi$ is an almost total Turing functional,
the image under $\Phi$ of a $\Pi_{1}^{0}$ class is also a $\Pi_{1}^{0}$
class.'' This is not true unless the $\Pi_{1}^{0}$ class is a subset
of the domain of $\Phi$. Fortunately, their proof only uses the $\Pi_{1}^{0}$
set $2^{\mathbb{N}}\smallsetminus U_{i}$, where $U_{i}$ is the $i$th
level of the optimal Martin-Löf test. This set is contained in the
domain of $\Phi$. Moreover, the proof does not rely on the compactness
of $2^{\mathbb{N}}$ at all, just on the effective compactness of
$K_{i}$. Therefore, no-randomness-from-nothing applies to all computable
probability spaces, not just the compact ones, as observed by Hoyrup
and Rojas \cite[Prop.~5]{Hoyrup:2009pi}.} Both these results, together, say that $F$ is a surjective map from
the set of $\mu$-Martin-Löf randoms to the set of $\nu$-Martin-Löf
randoms. Theorem~\ref{thm:ML-rp-nrfn} also holds for other computable
probability spaces with layerwise computable maps (Hoyrup and Rojas
\cite[Prop.~5]{Hoyrup:2009pi}). (Also see Hertling and Weihrauch
\cite{Hertling.Weihrauch:1998} for a randomness conservation result
for partial maps between effective topological spaces.)

Theorem~\ref{thm:ML-rp-nrfn} is sufficient for proving many of the
results which characterize Martin-Löf randomness for one probability
space in terms of Martin-Löf randomness for another.\footnote{In some applications (e.g.\ Hoyrup and Rojas \cite[Cor.~2]{Hoyrup:2009pi})
one may also need to apply the following theorem: if $\nu$ is absolutely
continuous with respect to $\mu$ and the density function $d\nu/d\mu$
is bounded from above by a constant (or by an $L^{1}(\mu)$-computable
function), then every $\nu$-random is also $\mu$-random.}  There are many such examples in Martin-Löf random Brownian motion
\cite[Cor.~2]{Hoyrup:2009pi}, \cite{Allen:aa}, \cite{Fouche2000}.

Bienvenu and Porter \cite{BienvenuSubmitted} and independently Rute
\cite{rute1} showed that randomness conservation does not hold for
computable randomness \cite[Thm.~4.2]{BienvenuSubmitted}\cite[Cor.~9.7]{rute1},
but it does hold for Schnorr randomness \cite[Thm.~4.1]{BienvenuSubmitted}\cite[Prop.~7.7]{rute1}.
Bienvenu and Porter asked if no-randomness-from-nothing holds for
Schnorr and computable randomness.

In Section~\ref{sec:NRFN-CR}, we show that no-randomness-from-nothing
holds for computable randomness.

In Section~\ref{sec:3-generalizations}, we generalize the results
of Section~\ref{sec:NRFN-CR} in three ways: First, we generalize
from almost-everywhere computable maps to Schnorr layerwise computable
maps (a form of effectively measurable map well-suited for computable
measure theory). Second, we generalize from Cantor space $2^{\mathbb{N}}$
to an arbitrary computable metric space. Third, we sketch how to relativize
the result to an oracle, except that we use uniform relativization
to which computable randomness is better suited. Section~\ref{sec:3-generalizations}
is independent from the rest of the paper.

In Section~\ref{sec:Application-of-nrfn}, we give an interesting
application of no-randomness-from-nothing for computable randomness.
We show that if a probability measure $\mu$ is the sum of a computable
sequence of measures $\mu_{n}$, then $x$ is $\mu$-computably random
if and only if $x$ is $\mu_{n}$-computably random for some $n$.

In Section~\ref{sec:NRFN-SR} we show no-randomness-from-nothing
does not hold for Schnorr randomness. We even show something stronger.
If $x$ is not computably random for $(2^{\mathbb{N}},\mu)$, then
there exists an almost-everywhere computable, measure-preserving map
$T\colon(2^{\mathbb{N}},\lambda)\rightarrow(2^{\mathbb{N}},\mu)$
such that $T^{-1}(\{x\})=\varnothing$.

In Section~\ref{sec:other-rand-notions} we complete the picture
by providing proofs of randomness conservation for difference randomness
(unpublished result of Bienvenu) and 2-randomness.

Last, in Section~\ref{sec:Characterizing-Martin-L=0000F6f-random},
we will show how randomness conservation and no-randomness-from-nothing
can be used to characterize a variety of randomness notions. The main
result is that Martin-Löf randomness is the weakest randomness notion
satisfying both randomness conservation and no-randomness-from-nothing.
We give two different formulations of this result, one for all computable
probability measures, and one for just the fair-coin probability measure.
The second relies on a recent result of Petrovi\'{c} \cite{Petrovic:}.

\subsection{Conclusions on Schnorr and computable randomness}

We caution the reader not to come to the hasty conclusion that Schnorr
randomness and computable randomness are ``unnatural'' just because
computable randomness does not satisfy randomness conservation and
Schnorr randomness does not satisfy no-randomness-from-nothing.

Indeed there is already compelling evidence to their naturalness.
Both Schnorr randomness and computable randomness have been characterized
by a number of theorems in analysis \cite{Gacs2011,Pathak:2014fk,Rute:2013pd}.
Moreover, as argued by Schnorr \cite[last paragraph]{Schnorr1971}
and Rute \cite{Rute:aa}, Schnorr randomness is the randomness notion
implicit in constructive measure theory. Last, Schnorr randomness
seems to be the weakest randomness notion sufficient for working with
measure theoretic objects (see, for example \cite{Gacs2011,Pathak:2014fk,Rute:2013pd}).

Instead, as we will show in a future paper \cite{Rute:mz}, it is
randomness conservation and no-randomness-from-nothing that need to
be modified. If one restricts the measure-preserving maps to those
where the ``conditional probability'' is computable, then one recovers
both randomness conservation and no-randomness-from-nothing for Schnorr
and computable randomness. This class of maps is natural and covers
nearly every measure-preserving map used in practice --- including
isomorphisms, projections on product measures, and even the von Neumann
coin example above. Martin-Löf randomness also behaves better under
these maps. Indeed, randomness conservation, no-randomness-from-nothing,
and van Lambalgen's theorem can be combined into one unified theorem
for Schnorr and Martin-Löf randomness.

\subsection{Status of randomness conservation and no-randomness-from-nothing}

We end this introduction with a table summarizing the known results
about randomness conservation and no-randomness-from-nothing. 

\bigskip 

\noindent %
\begin{tabular}{lclcl}
\toprule 
\multirow{1}{*}{Randomness notion} & \multicolumn{2}{l}{Randomness conservation\footnotemark} & \multicolumn{2}{l}{No-randomness-from-nothing}\tabularnewline
\midrule
Kurtz random & Yes & \cite[Prop.~7.7]{rute1} & \strong{No} & Thm.~\ref{thm:not-CR-nrnf}\tabularnewline
Schnorr random & Yes & \cite[Thm.~4.1]{BienvenuSubmitted}\cite[Prop.~7.7]{rute1} & \strong{No} & Thm.~\ref{thm:not-CR-nrnf}\tabularnewline
computable random & \strong{No} & \cite[Thm.~4.2]{BienvenuSubmitted}\cite[Cor.~9.7]{rute1} & Yes & Thm.~\ref{thm:CR-nrfn}\tabularnewline
Martin-Löf random & Yes & \cite[Thm.~3.2]{BienvenuSubmitted} & Yes & \cite[Thm.~3.5]{BienvenuSubmitted}\tabularnewline
Difference random & Yes & Prop.~\ref{prop:Bienvenu} & Yes & Prop.~\ref{prop:Bienvenu}\tabularnewline
Demuth random & Yes & Folklore & \strong{?} & \tabularnewline
weak 2-random & Yes & \cite[Thm.~5.9]{Bienvenu.Holzl.Porter.ea:} & Yes & \cite[Thm.~6.18]{Bienvenu.Holzl.Porter.ea:}\tabularnewline
2-random & Yes & Prop.~\ref{prop:2-rand-rp-nrfn} & Yes & Prop.~\ref{prop:2-rand-rp-nrfn}\tabularnewline
\bottomrule
\end{tabular}\footnotetext{When it is true, randomness conservation is easy to prove, and in many cases well known. Therefore, the positive results in this column should probably be attributed to folklore. The citations given are for reference.}\bigskip

\subsection{Acknowledgements}

We would like to thank Laurent Bienvenu for pointing us to the results
on difference randomness, Demuth randomness, and weak 2-randomness.
I would also like to thank both referees for their thorough reviews.

\section{\label{sec:Definitions-and-notation}Definitions and notation}

Let $2^{\mathbb{N}}$ denote Cantor space and $2^{*}$ the set of
all finite binary strings. Let $\varepsilon$ be the empty string,
and $[\sigma]$ the cylinder set of $\sigma\in2^{*}$. For a finite
Borel measure $\mu$ on $2^{\mathbb{N}}$ we will use the notation
$\mu(\sigma):=\mu([\sigma])$. For a finite Borel measure $\mu$ on
$2^{\mathbb{N}}\times2^{\mathbb{N}}$ we will use the notation $\mu(\sigma\times\tau):=\mu([\sigma]\times[\tau])$.
A measure $\mu$ on $2^{\mathbb{N}}$ is \emph{computable} if $\sigma\mapsto\mu(\sigma)$
is computable. The \emph{fair-coin measure} $\lambda$ is given by
$\lambda(\sigma)=2^{-|\sigma|}$.

Given a computable map $F\colon2^{\mathbb{N}}\rightarrow2^{\mathbb{N}}$,
the \emph{pushforward} of $\mu$ along $F$ is the computable measure
$\mu_{F}$ given by $\mu_{F}(\sigma)=\mu(F^{-1}([\sigma]))$.

As usual, we naturally identify the spaces $2^{\mathbb{N}}\times2^{\mathbb{N}}$
and $2^{\mathbb{N}}$, via the computable isomorphism $(x,y)\mapsto x\oplus y$.
(Here $x\oplus y$ is the sequence $z\in2^{\mathbb{N}}$ given by
$z(2n)=x(n)$ and $z(2n+1)=y(n)$.) We also identify their computable
measures, where $\mu$ on $2^{\mathbb{N}}\times2^{\mathbb{N}}$ is
identified with the pushforward of $\mu$ along $(x,y)\mapsto x\oplus y$.

We define Martin-Löf randomness, computable randomness, and Schnorr
randomness through an integral test characterization. These characterizations
are due to Levin \cite{Levin:1976uq}, Rute \cite[Thms.~5.3, 5.8]{rute1},
and Miyabe \cite[Thm.~3.5]{Miyabe:2013uq} respectively. Recall that
a \emph{lower semicomputable} function $t\colon2^{\mathbb{N}}\rightarrow[0,\infty]$
is the sum of a computable sequence of computable functions $t_{n}\colon2^{\mathbb{N}}\rightarrow[0,\infty)$.
\begin{defn}
\label{def:randomness}Let $\mu$ be a computable measure on $2^{\mathbb{N}}$
and let $x\in2^{\mathbb{N}}$.
\begin{enumerate}
\item $x$ is \emph{$\mu$-Martin-Löf random} if $t(x)<\infty$ for all
lower semicomputable functions $t\colon2^{\mathbb{N}}\rightarrow[0,\infty]$
such that 
\[
\int t\,d\mu\leq1.
\]

\item $x$ is \emph{$\mu$-computably random} if $t(x)<\infty$ for all
lower semicomputable functions $t\colon2^{\mathbb{N}}\rightarrow[0,\infty]$
and all computable probability measures $\nu$ on $2^{\mathbb{N}}$
such that 
\begin{equation}
\int_{[\sigma]}t\,d\mu\leq\nu(\sigma)\qquad(\sigma\in2^{*}).\label{eq:test-pair1}
\end{equation}

\item $x$ is \emph{$\mu$-Schnorr random} if $t(x)<\infty$ for all lower
semicomputable functions $t\colon2^{\mathbb{N}}\rightarrow[0,\infty]$
such that 
\[
\int t\,d\mu=1.
\]

\end{enumerate}
\end{defn}
From these definitions it is obvious that Martin-Löf randomness implies
computable randomness implies Schnorr randomness. (It is also known
that the implications do not reverse.) Also, $x\in2^{\mathbb{N}}$
is \emph{$\mu$-Kurtz random} if $x$ is not in any $\Sigma_{2}^{0}$
$\mu$-null set. Every Schnorr random is Kurtz random.

Our definition of computable randomness transfers to $2^{\mathbb{N}}\times2^{\mathbb{N}}$
as follows.
\begin{prop}
\label{prop:CR-2-dim}Let $\mu$ be a computable measure on $2^{\mathbb{N}}\times2^{\mathbb{N}}$.
A pair $(x,y)\in2^{\mathbb{N}}\times2^{\mathbb{N}}$ is $\mu$-computably
random if and only if $t(x,y)<\infty$ for all lower semicomputable
functions $t\colon2^{\mathbb{N}}\times2^{\mathbb{N}}\rightarrow[0,\infty]$
and all computable probability measures $\nu$ on $2^{\mathbb{N}}\times2^{\mathbb{N}}$
such that 
\begin{equation}
\int_{[\sigma]\times[\tau]}t\,d\mu\leq\nu(\sigma\times\tau)\qquad(\sigma,\tau\in2^{*}).\label{eq:test-pair2}
\end{equation}
\end{prop}
\begin{proof}
Let $\mu'$ denote the pushforward of $\mu$ along $(x,y)\mapsto x\oplus y$.
Given any test pair $t,\nu$ satisfying (\ref{eq:test-pair2}) with
$\mu$, consider the test pair $t',\nu'$ where $t'(x\oplus y)=t(x,y)$
and $\nu'$ is the pushforward of $\nu$ under the map $(x,y)\mapsto x\oplus y$.
Then $t',\nu'$ satisfies (\ref{eq:test-pair1}) with $\mu'$. Conversely,
any test pair $t,\nu$ satisfying (\ref{eq:test-pair1}) can be translated
into a test pair $t,\nu$ satisfying (\ref{eq:test-pair2}) with $\mu$.
\end{proof}
The following more classical definition of computable randomness will
be useful as well.
\begin{lem}[{See Rute \cite[Def~2.4]{rute1}}]
\label{lem:CR-standard-def}If $\mu$ is a computable measure on
$2^{\mathbb{N}}$, a sequence $x\in2^{\mathbb{N}}$ is $\mu$-computably
random if and only if both for all $n$, $\mu(x{\upharpoonright_{n}})>0$
and for all computable measures $\nu$, 
\[
\liminf_{n}\frac{\nu(x{\upharpoonright_{n}})}{\mu(x{\upharpoonright_{n}})}<\infty.
\]

\end{lem}
The ratio $\nu(x{\upharpoonright_{n}})/\mu(x{\upharpoonright_{n}})$
is known as a \emph{martingale} and can be thought of as a fair betting
strategy. (See Rute \cite[\S2]{rute1} for more discussion.) By an
effective version of Doob's martingale convergence theorem, this ratio
converges on computable randoms.
\begin{lem}[{Folklore \cite[Thm~7.1.3]{Downey2010}\footnote{The proof in \cite[Thm~7.1.3]{Downey2010} is for when $\mu$ is the
fair-coin measure, but the proof is the same for all computable measures.}}]
\label{lem:doob}Assume $x\in2^{\mathbb{N}}$ is $\mu$-computably
random and $\nu$ is a computable measure. Then the following limit
converges, 
\[
\lim_{n}\frac{\nu(x{\upharpoonright_{n}})}{\mu(x{\upharpoonright_{n}})}.
\]

\end{lem}

\begin{defn}
A partial map $T\colonsubseteq2^{\mathbb{N}}\rightarrow2^{\mathbb{N}}$
is said to be \emph{$\mu$-almost-everywhere ($\mu$-a.e.)~computable}
for a computable probability measure $\mu$ if $T$ is partial computable\footnote{For concreteness, say $T\colonsubseteq2^{\mathbb{N}}\rightarrow2^{\mathbb{N}}$
is \emph{partial computable} if it is given by a monotone machine
$M\colonsubseteq2^{*}\rightarrow2^{*}$. Say $x\in\operatorname{dom}T$
if and only if there is some $y\in2^{\mathbb{N}}$ such that $y=\lim_{n}M(x\upharpoonright_{n})$.
In this case, $T(x)=y$. The domain of a partial computable map is
always $\Pi_{2}^{0}$.} and $\mu(\operatorname{dom}T)=1$.
\end{defn}
We denote $\mu$-a.e.~computable maps with the notation $T\colon(2^{\mathbb{N}},\mu)\rightarrow2^{\mathbb{N}}$.
Moreover, given a $\mu$-a.e.~computable map $T\colon(2^{\mathbb{N}},\mu)\rightarrow2^{\mathbb{N}}$,
the pushforward measure $\mu_{T}$ (of $\mu$ along $T$) is a well-defined
probability measure computable from $T$ and $\mu$. We use the notation
$T\colon(2^{\mathbb{N}},\mu)\rightarrow(2^{\mathbb{N}},\nu)$ to denote
that $T$ is \emph{measure-preserving}, that is $\nu=\mu_{T}$.

\section{\label{sec:NRFN-CR}No-randomness-from-nothing for computable randomness}

In this section we will prove the following.
\begin{thm}[No-randomness-from-nothing for computable randomness.]
\label{thm:CR-nrfn}If $\mu$ is a computable probability measure
on $2^{\mathbb{N}}$, $T\colon(2^{\mathbb{N}},\mu)\rightarrow2^{\mathbb{N}}$
is a $\mu$-a.e.~computable map, and $y\in2^{\mathbb{N}}$ is $\mu_{T}$-computably
random, then $y=T(x)$ for some $\mu$-computably random $x\in2^{\mathbb{N}}$.
\end{thm}
The proof will be similar to that of van Lambalgen's theorem \cite[\S6.9.1]{Downey2010}\cite[Thm~3.4.6]{Nies2009},
which states that $(x,y)$ is Martin-Löf random if and only if $x$
is Martin-Löf random and $y$ is Martin-Löf random relative to $x$.
First, however, we require a number of lemmas establishing properties
of computable randomness on $2^{\mathbb{N}}$ and $2^{\mathbb{N}}\times2^{\mathbb{N}}$.
The following lemma establishes randomness conservation for computable
randomness along a.e.~computable isomorphisms and will be a key tool
in this proof.\footnote{We will show, in Theorem~\ref{thm:slwc-isomorphism}, that this lemma
also holds for Schnorr layerwise computable maps.}
\begin{lem}[{Rute \cite[Prop~7.6, Thm.~7.11]{rute1}}]
\label{lem:isomorphism}Let $\mu$ and $\nu$ be computable probability
measures on $2^{\mathbb{N}}$. Let $F\colon(2^{\mathbb{N}},\mu)\rightarrow(2^{\mathbb{N}},\nu)$
and $G\colon(2^{\mathbb{N}},\nu)\rightarrow(2^{\mathbb{N}},\mu)$
be almost-everywhere computable measure-preserving maps such that
\[
G(F(x))=x\ \ \mu\text{-a.e.}\qquad\text{and}\qquad F(G(y))=y\ \ \nu\text{-a.e.}
\]
Then $F$ and $G$ both conserve computable randomness, and if $x$
is $\mu$-computably random and $y$ is $\nu$-computably random then
\[
G(F(x))=x\qquad\text{and}\qquad F(G(y))=y.
\]

\end{lem}

However, most maps $T\colon(2^{\mathbb{N}},\mu)\rightarrow2^{\mathbb{N}}$
are not isomorphisms. Nonetheless, we can turn them into isomorphisms
by mapping $x$ to the pair $(x,T(x))$.\footnote{We will show, in Lemma~\ref{lem:graph-slwc}, that this lemma also
holds for Schnorr layerwise computable maps.}
\begin{lem}
\label{lem:graph}Let $\mu$ be a computable probability measure on
$2^{\mathbb{N}}$, let $T\colon(2^{\mathbb{N}},\mu)\rightarrow2^{\mathbb{N}}$
be a $\mu$-a.e.~computable map, and let $(\mathrm{id},T)$ be the
map $x\mapsto(x,T(x))$. If $(x,y)$ is $\mu_{(\mathrm{id},T)}$-computably
random, then $x$ is $\mu$-computably random and $y=T(x)$.\end{lem}
\begin{proof}
Clearly $(\mathrm{id},T)$ is a $\mu$-a.e.~computable map. Moreover,
notice that $(\mathrm{id},T)$ and its inverse $(x,y)\mapsto x$ satisfy
the conditions of Lemma~\ref{lem:isomorphism}. Therefore if $(x,y)$
is $\mu_{(\mathrm{id},T)}$-computably random, then $x$ is $\mu$-computably
random. By the composition $(x,y)\mapsto x\mapsto(x,T(x))$, we have
that $y=T(x)$.
\end{proof}
The main idea of the proof of Theorem~\ref{thm:CR-nrfn} is as follows.
The measure $\mu_{(\mathrm{id},T)}$ is supported on the graph of
$T$, 
\[
\{(x,T(x)):x\in2^{\mathbb{N}}\}.
\]
Therefore, given a $\mu_{T}$-computably random $y$, by this last
lemma, it is sufficient to find some $x$ which makes $(x,y)$ $\mu_{(\mathrm{id},T)}$-computably
random. That is the goal of the rest of this section.
\begin{lem}
\label{lem:cond-measure}Let $\mu$ be a computable probability measure
on $2^{\mathbb{N}}\times2^{\mathbb{N}}$ with second marginal $\mu_{2}$
(that is $\mu_{2}(\tau)=\mu(\varepsilon\times\tau)$). Assume $y\in2^{\mathbb{N}}$
is $\mu_{2}$-computably random. The following properties hold.
\begin{enumerate}
\item For a fixed $\tau\in2^{*}$,\textup{ $\mu(\cdot\times\tau)$ is a
measure, that is
\[
\mu(\sigma0\times\tau)+\mu(\sigma1\times\tau)=\mu(\sigma\times\tau)\qquad(\sigma\in2^{*}).
\]
}
\item For a fixed $\sigma\in2^{*}$, \textup{$\mu(\sigma\times\cdot)$ is
a measure.}
\item The following limit converges for each $\sigma\in2^{*}$, 
\[
\mu(\sigma|y):=\lim_{n}\frac{\mu(\sigma\times y{\upharpoonright_{n}})}{\mu_{2}(y{\upharpoonright_{n}})}.
\]

\item The function $\mu(\cdot|y)$ defines a probability measure, that is
$\mu(\varepsilon|y)=1$ and 
\[
\mu(\sigma0|y)+\mu(\sigma1|y)=\mu(\sigma|y)\qquad(\sigma\in2^{*}).
\]

\item For a continuous function $f\colon2^{\mathbb{N}}\times2^{\mathbb{N}}\rightarrow\mathbb{R}$,
if $f^{y}=f(\cdot,y)$ then
\[
\int f^{y}\,d\mu(\cdot|y)=\lim_{n}\frac{\int_{[\varepsilon]\times[y{\upharpoonright_{n}}]}f\,d\mu}{\mu_{2}(y{\upharpoonright_{n}})}.
\]

\item For a nonnegative lower semicontinuous function $t\colon2^{\mathbb{N}}\times2^{\mathbb{N}}\rightarrow[0,\infty]$,
if $t^{y}=t(\cdot,y)$ then
\[
\int t^{y}\,d\mu(\cdot|y)\leq\lim_{n}\frac{\int_{[\varepsilon]\times[y{\upharpoonright_{n}}]}t\,d\mu}{\mu_{2}(y{\upharpoonright_{n}})}.
\]

\end{enumerate}
\end{lem}
\begin{proof}
(1) and (2) are apparent. Then (3) follows from (2) and Lemma~\ref{lem:doob}.
Then (4) follows from (1) and the definition of $\mu(\cdot|y)$.

As for (5), first consider the case where $f$ is a step function
of the form $\sum_{i=0}^{k}a_{i}\mathbf{1}_{[\sigma_{i}]}$. This
case follows from (4). Since such step functions are dense in the
continuous functions under the norm $\|f\|=\sup_{x}f(x)$, we have
the result.

As for (6), $t=\sum_{k}f_{k}$ for a sequence of continuous nonnegative
$f_{k}$. Then we apply the monotone convergence theorem (MCT) for
integrals and Fatou's lemma for sums,
\begin{multline*}
\int t^{y}\,d\mu(\cdot|y)=\int\sum_{k}f_{k}^{y}\,d\mu(\cdot|y)\overset{\textrm{MCT}}{=}\sum_{k}\int f_{k}^{y}\,d\mu(\cdot|y)\\
=\sum_{k}\lim_{n}\frac{\int_{[\varepsilon]\times[y{\upharpoonright_{n}}]}f_{k}\,d\mu}{\mu_{2}(y{\upharpoonright_{n}})}\overset{\textrm{Fatou}}{\leq}\lim_{n}\sum_{k}\frac{\int_{[\varepsilon]\times[y{\upharpoonright_{n}}]}f_{k}\,d\mu}{\mu_{2}(y{\upharpoonright_{n}})}=\lim_{n}\frac{\int_{[\varepsilon]\times[y{\upharpoonright_{n}}]}t\,d\mu}{\mu_{2}(y{\upharpoonright_{n}})}.\qedhere
\end{multline*}

\end{proof}

\begin{defn}[{\cite[Def.~2.1]{Kjos-Hanssen:2010aa}\cite[Def.~5.37]{Bienvenu2011}}]
Let $\mu$ be a probability measure on $2^{\mathbb{N}}$ which may
not be computable. A sequence $x\in2^{\mathbb{N}}$ is \emph{blind $\mu$-Martin-Löf random}
(or \emph{$\mu$-Hippocratic random}) \emph{relative to $y\in2^{\mathbb{N}}$}
if $t(x)<\infty$ for all $t\colon2^{\mathbb{N}}\rightarrow[0,\infty]$
which are lower semicomputable relative to $y$ such that $\int t\,d\mu<\infty$.
\end{defn}
Finally, we have the tools to find some $x$ to pair with $y$. What
follows is a variation of van Lambalgen's theorem, similar to that
given by Takahashi \cite[Thm.~5.2]{Takahashi:2008}.
\begin{lem}
\label{lem:vL}Let $\mu$ be a computable measure on $2^{\mathbb{N}}\times2^{\mathbb{N}}$.
Let $y$ be $\mu_{2}$-computably random. Let $x$ be blind $\mu(\cdot|y)$-Martin-Löf
random relative to $y$. Then $(x,y)$ is $\mu$-computably random.\end{lem}
\begin{proof}
Assume $y$ is $\mu_{2}$-computably random, but $(x,y)$ is not $\mu$-computably
random. Then by Proposition~\ref{prop:CR-2-dim} there is a lower
semicomputable function $t$ and a computable measure $\nu$ such
that $t(x,y)=\infty$ and $\int_{[\sigma]\times[\tau]}t\,d\mu\leq\nu(\sigma\times\tau)$.

Let $t^{y}=t(\cdot,y)$. Then $t^{y}$ is lower semicomputable relative
to $y$. Moreover, since $t$ is lower semicontinuous, we have by
Lemma~\ref{lem:cond-measure} that 
\[
\int t^{y}\,d\mu(\cdot|y)\leq\lim_{n}\frac{\int_{[\varepsilon]\times[y{\upharpoonright_{n}}]}t\,d\mu}{\mu_{2}(y{\upharpoonright_{n}})}\leq\lim_{n}\frac{\nu(\varepsilon\times y{\upharpoonright_{n}})}{\mu_{2}(y{\upharpoonright_{n}})}
\]
where the right-hand side converges to a finite value by Lemma~\ref{lem:doob}
since $y$ is $\mu_{2}$-computably random. Therefore, $x$ is not
blind Martin-Löf random relative to $y$ since $\int t^{y}\,d\mu(\cdot|y)<\infty$
and $t^{y}(x)=t(x,y)=\infty$.
\end{proof}
The proof of Theorem~\ref{thm:CR-nrfn} easily follows.
\begin{proof}[Proof of Theorem~\ref{thm:CR-nrfn}.]
Let $\mu$ be a computable probability measure on $2^{\mathbb{N}}$,
let $T\colon(2^{\mathbb{N}},\mu)\rightarrow2^{\mathbb{N}}$ be an
a.e.~computable map, and let $y\in2^{\mathbb{N}}$ be $\mu_{T}$-computably
random. We want to find some $\mu$-computably random $x\in2^{\mathbb{N}}$
such that $y=T(x)$.

Let $\nu=\mu_{(\mathrm{id},T)}$, and let $x$ be blind $\nu(\cdot|y)$-Martin-Löf
random relative to $y$ (there are $\nu(\cdot|y)$-measure-one many,
so there is at least one). By Lemma~\ref{lem:vL}, $(x,y)$ is $\mu_{(\mathrm{id},T)}$-computably
random. By Lemma~\ref{lem:graph}, $x$ is computably random and
$y=T(x)$.
\end{proof}

\section{Three generalizations to theorem~\ref{thm:CR-nrfn}\label{sec:3-generalizations}}

Algorithmic randomness is becoming more focused around the ideas in
probability theory and measure theory. Therefore the tools in algorithmic
randomness need to be able to handle a larger variety of maps, a larger
variety of spaces, and even a larger variety of relativizations. In
this section we generalize Theorem~\ref{thm:CR-nrfn} to Schnorr
layerwise computable maps, arbitrary computable metric spaces, and
uniform relativization. This section is independent of the later sections
and may be skipped.

\subsection{Generalizing the types of maps}

The main results of this subsection are generalizations of Theorem~\ref{thm:CR-nrfn}
and Lemma~\ref{lem:isomorphism} to Schnorr layerwise computable
maps. Before, proving the theorems, let us introduce Schnorr layerwise
computable maps, and explain why they are important.

\subsubsection{Schnorr layerwise computable maps}

So far, the results in this paper have been phrased in terms of almost-everywhere
computable maps. These maps are easy for a computability theorist
to understand and they are sufficient for many purposes. However,
almost-everywhere computable maps are only almost-everywhere continuous
and therefore are not adequate computable representations of the measurable
maps found in probability theory.\footnote{For example, the map $F\colon(2^{\mathbb{N}},\lambda)\rightarrow2$
which takes a sequence $x\in2^{\mathbb{N}}$ and returns $1$ if $\sup_{n}\frac{1}{n}\sum_{k=0}^{n-1}x(n)>2/3$
and $0$ otherwise is not almost-everywhere continuous, and therefore
not almost-everywhere computable (even relative to an oracle). (We
thank Bjørn Kjos-Hanssen for this example.)}

Instead, we need a notion of an ``effectively $\mu$-measurable function''
$F\colon(2^{\mathbb{N}},\mu)\rightarrow2^{\mathbb{N}}$. There are
many approaches in the literature. One approach, dating back to the
Russian constructivist Šanin \cite[\S 15.4]{Sanin:1968dq}, is to
use the topology of convergence in measure, which is metrizable via
many equivalent metrics, including $\rho(F,G)=\int d(F(x),G(x))\,d\mu(x)$
where $d$ is the usual metric on $2^{\mathbb{N}}$. (This is a modification
of the usual $L^{1}$-metric.\footnote{The space $(2^{\mathbb{N}},d)$ has bounded diameter. For a general
codomain $(\mathbb{Y},d)$, use the metric $\rho(F,G)=\int\min\{d(F(x),G(x)),1\}\,d\mu(x)$.}) If $\mu$ is a Borel measure, the space $L^{0}(2^{\mathbb{N}},\mu;2^{\mathbb{N}})$
of measurable functions $F\colon(2^{\mathbb{N}},\mu)\rightarrow2^{\mathbb{N}}$
modulo $\mu$-almost-everywhere equivalence is a Polish space.
\begin{defn}
Fix a computable probability measure $\mu$. A function $F\in L^{0}(2^{\mathbb{N}},\mu;2^{\mathbb{N}})$
is \emph{effectively measurable} if there is a computable sequence
of a.e.~computable functions $F_{n}\colon(2^{\mathbb{N}},\mu)\rightarrow2^{\mathbb{N}}$
such that $\rho(F_{n},F_{m})\leq2^{-m}$ for all $n>m$ and $F=\lim_{n}F_{n}$
(where the limit is in the metric $\rho$).

Let $\widetilde{F}$ be the pointwise limit of $F_{n}$ (when it converges).\footnote{The sequence $F_{n}$ converges $\mu$-almost-everywhere since $F_{n}$
converges at a geometric rate of convergence in $\rho$. Also $\widetilde{F}$
is possibly partial, since there may be a measure zero set of $x\in2^{\mathbb{N}}$
where $\lim_{n}F_{n}(x)$ does not converge. }  Call $\widetilde{F}$ the \emph{canonical representative} of $F$.
\end{defn}
Surprisingly the canonical representative is always defined on Schnorr
randoms. Moreover, if $F$ and $G$ are $\mu$-a.e.~equal effectively
measurable maps, then $\widetilde{F}(x)=\widetilde{G}(x)$ on $\mu$-Schnorr
randoms $x$ \cite[p.~41, Prop.~3.18]{Rute:2013pd} (also see \cite[Thm.~3.9]{Pathak:2014fk}).
Finally, Rute \cite[p.~41, Prop~3.23]{Rute:2013pd} showed that these
representative functions $\widetilde{F}$ are equivalent to the Schnorr
layerwise computable functions of Miyabe \cite{Miyabe:2013uq}. (This
equivalence is an effective version of Lusin's theorem.)
\begin{defn}
\label{def:Sch-layerwise}A measurable map $F\colon(\mathbb{X},\mu)\rightarrow\mathbb{Y}$
is \emph{Schnorr layerwise computable} if there is a computable sequence
of effectively closed (that is $\Pi_{1}^{0}$) sets $C_{n}\subseteq\mathbb{X}$
and a computable sequence of computable functions $F_{n}\colon C_{n}\rightarrow\mathbb{Y}$
such that
\begin{enumerate}
\item $\mu(C_{n})\geq1-2^{-n}$ and $\mu(C_{n})$ is computable in $n$,
and
\item $F_{n}=F\upharpoonright C_{n}$ for all $n$.
\end{enumerate}
\end{defn}
The remainder of the results in this section will be expressed in
terms of Schnorr layerwise computability.

\subsubsection{No-randomness-from-nothing for computable randomness and Schnorr
layerwise computable functions}

In order to extend Theorem~\ref{thm:CR-nrfn} to Schnorr layerwise
computable functions, it suffices to prove a Schnorr layerwise computable
version of Lemma~\ref{lem:graph}. In order to do that, we need some
lemmas. 

Notice that the definitions of Schnorr randomness and computable randomness
naturally extend to all computable measures on $2^{\mathbb{N}}$,
not just probability measures. This saves us the step of having to
normalize a measure into a probability measure, as in this next lemma.
\begin{lem}
\label{lem:CR-conditioned-on-subset}Let $C$ be an effectively closed
subset of $2^{\mathbb{N}}$ such that $\mu(C)$ is computable and
positive. Let $\nu$ be the measure $\nu(\sigma)=\mu(C\cap[\sigma])$.
This measure is a computable measure, and for any $x\in C$, if $x$
is $\mu$-computably random then $x$ is $\nu$-computably random.\end{lem}
\begin{proof}
Since $\mu(C)$ is computable we can compute $\mu(C\cap[\sigma])$.
By the Lebesgue density theorem, we have for almost-every $x\in C$
that 
\[
\lim_{n}\frac{\mu(C\cap[x{\upharpoonright_{n}}])}{\mu(x{\upharpoonright_{n}})}=1.
\]
Moreover, the set 
\[
N=\left\{ x\in2^{\mathbb{N}}:x\in C\ \textrm{and}\ \limsup_{n}\frac{\mu(C\cap[x{\upharpoonright_{n}}])}{\mu(x{\upharpoonright_{n}})}<1\right\} 
\]
is a $\Sigma_{2}^{0}$ set of $\mu$-measure $0$. Therefore, $N$
does not contain any $\mu$-Kurtz randoms. Hence if $x\in C$ and
$x$ is $\mu$-computably random (so $\mu$-Kurtz random), then $\limsup_{n}\frac{\mu(C\cap[x{\upharpoonright_{n}}])}{\mu(x{\upharpoonright_{n}})}=1$.\footnote{Actually, one can show that the limit is $1$ for all $\mu$-Schnorr
randoms \cite[p.~51, Thm~6.3]{Rute:2013pd}, and hence for all $\mu$-computable
randoms, but this is not needed.} Now, let $\rho$ be a computable measure. Then for any $\mu$-computable
random $x\in C$, we have that $\rho(x{\upharpoonright_{n}})/\mu(x{\upharpoonright_{n}})$
converges to a finite number (Lemma~\ref{lem:doob}). Hence 
\begin{align*}
\liminf_{n}\frac{\rho(x{\upharpoonright_{n}})}{\nu(x{\upharpoonright_{n}})} & =\liminf_{n}\frac{\rho(x{\upharpoonright_{n}})}{\mu(C\cap[x{\upharpoonright_{n}}])}\\
 & =\liminf_{n}\frac{\rho(x{\upharpoonright_{n}})}{\mu(x{\upharpoonright_{n}})}\cdot\frac{\mu(x{\upharpoonright_{n}})}{\mu(C\cap[x{\upharpoonright_{n}}])}\\
 & =\lim_{n}\frac{\rho(x{\upharpoonright_{n}})}{\mu(x{\upharpoonright_{n}})}\cdot\liminf_{n}\frac{\mu(x{\upharpoonright_{n}})}{\mu(C\cap[x{\upharpoonright_{n}}])}\\
 & =\lim_{n}\frac{\rho(x{\upharpoonright_{n}})}{\mu(x{\upharpoonright_{n}})}<\infty.
\end{align*}
Since $\rho$ is arbitrary, $x$ is $\nu$-computably random (Lemma~\ref{lem:CR-standard-def}).
\end{proof}

\begin{lem}
\label{lem:CR-smaller-to-bigger}Assume $\mu\leq\nu$ (that is $\mu(A)\leq\nu(A)$
for all measurable sets $A\subseteq2^{\mathbb{N}}$). Then if $x$
is $\mu$-computably random, then $x$ is $\nu$-computably random.\end{lem}
\begin{proof}
Assume $x$ is not $\nu$-computably random. Then there is an integral
test $t$ such that $t(x)=\infty$ and $\int_{A}t\,d\nu\leq\rho(A)$
for some computable measure $\rho$. Then $\int_{A}t\,d\mu\leq\int_{A}t\,d\nu\leq\rho(A)$
and therefore $t$ is a test on $\mu$ as well. Hence $x$ is not
$\mu$-computably random.
\end{proof}

\begin{lem}
\label{lem:injective-slwc}Assume $T\colon(2^{\mathbb{N}},\mu)\rightarrow2^{\mathbb{N}}$
is an injective Schnorr layerwise computable map. Then the pushforward
measure $\mu_{T}$ is computable, and $T$ has an injective Schnorr
layerwise computable inverse $S$. Moreover, both $T$ and $S$ conserve
computable randomness.\end{lem}
\begin{proof}
The computability of the pushforward is proved by Hoyrup and Rojas
\cite[Prop.~4]{Hoyrup:2009fe} and Rute \cite[p.~42, Prop.~3.25]{Rute:2013pd}. 

Let $C_{n}$ and $T_{n}\colon C_{n}\rightarrow2^{\mathbb{N}}$ be
the sequence of effectively closed sets and partial computable functions
as in Definition~\ref{def:Sch-layerwise}. Since $T_{n}$ is total
computable on $C_{n}$ (and $2^{\mathbb{N}}$ is effectively compact),
the image $D_{n}:=T_{n}(C_{n})$ is effectively closed in $n$. Since
$T$ is injective, $\mu_{T}(D_{n})=\mu(T^{-1}(D_{n}))=\mu(T_{n}^{-1}(D_{n}))=\mu(C_{n})$.
For a fixed $n$, we can compute a total computable map $S_{n}\colon D_{n}\rightarrow C_{n}$
given as follows. For each $y\in D_{n}$, we can compute the set $P_{y}=\{x\in C_{n}:T_{n}(x)=y\}$
as a $\Pi_{1}^{0}(y)$ set. (To see that it is $\Pi_{1}^{0}(y)$,
for each $x$, just wait until $x\notin C_{n}$ or $T_{n}(x){\upharpoonright_{k}}\neq y{\upharpoonright_{k}}$
for some $k$.) Since $T_{n}\colon C_{n}\rightarrow D_{n}$ is a bijection,
$P_{y}$ is a singleton $\{x\}$. Let $S_{n}(y)=x$. This $S_{n}$
is computable on $D_{n}$ (again using the compactness of $2^{\mathbb{N}}$).
Combining these $S_{n}$ we have a Schnorr layerwise computable map
$S\colon(2^{\mathbb{N}},\mu_{T})\rightarrow(2^{\mathbb{N}},\mu)$.
This map is injective (on its domain $\bigcup_{n}D_{n}$) since each
of the $S_{n}$ are injective. Moreover, it is the inverse of $T$
since each $S_{n}$ is the inverse of $T_{n}$.

Let $x$ be $\mu$-computably random. Since $x$ is $\mu$-Schnorr
random, there is some $C_{n}$ from the Schnorr layerwise description
of $T$ such that $x_{0}\in C_{n}$ and $T$ is computable on $C_{n}$.
(If $x$ is in no such $C_{n}$, then $x$ is in $\bigcap_{n}\left(2^{\mathbb{N}}\smallsetminus C_{n}\right)$
and cannot be Schnorr random.) Let $\mu_{n}$ be the measure given
by $\mu_{n}(\sigma)=\mu(C_{n}\cap[\sigma])$. By Lemma~\ref{lem:CR-conditioned-on-subset},
$\mu_{n}$ is computable and $x$ is $\mu_{n}$-computably random.
Notice that $T_{n}\colon(2^{\mathbb{N}},\mu_{n})\rightarrow2^{\mathbb{N}}$
is $\mu_{n}$-a.e.~computable and therefore $(\mu_{n})_{T}$ is computable.
Moreover, $T_{n}\colon(2^{\mathbb{N}},\mu_{n})\rightarrow(2^{\mathbb{N}},(\mu_{n})_{T})$
and $S_{n}\colon(2^{\mathbb{N}},(\mu_{n})_{T})\rightarrow(2^{\mathbb{N}},\mu_{n})$
are a.e.~computable measure-preserving inverses as in Lemma~\ref{lem:isomorphism}.
Therefore $T(x)=T_{n}(x)$ is $(\mu_{n})_{T}$-computably random.
Since $(\mu_{n})_{T}\leq\mu_{T}$, we have by Lemma~\ref{lem:CR-smaller-to-bigger},
$T(x)$ is $\mu_{T}$-computably random.
\end{proof}
Now we can give a Schnorr layerwise computable version of Lemma~\ref{lem:graph}.
\begin{lem}
\label{lem:graph-slwc}Let $\mu$ be a computable probability measure
on $2^{\mathbb{N}}$, let $T\colon(2^{\mathbb{N}},\mu)\rightarrow2^{\mathbb{N}}$
be a Schnorr layerwise computable map, and let $(\mathrm{id},T)$
be the map $x\mapsto(x,T(x))$. If $(x,y)$ is $\mu_{(\mathrm{id},T)}$-computably
random, then $x$ is $\mu$-computably random and $y=T(x)$.\end{lem}
\begin{proof}
Clearly $(\mathrm{id},T)$ is an injective Schnorr layerwise computable
map. By Lemma~\ref{lem:injective-slwc}, $(\mathrm{id},T)$ has an
injective Schnorr layerwise computable inverse map $S\colon(2^{\mathbb{N}},\mu_{(\mathrm{id},T)})\rightarrow(2^{\mathbb{N}},\mu)$.
This inverse is clearly the map $(x,y)\mapsto x$ (except restricted
to a smaller domain to guarantee injectivity). If $(x,y)$ is $\mu_{(\mathrm{id},T)}$-computably
random, then $x=S(x,y)$ is $\mu$-computably random by the randomness
conservation in Lemma~\ref{lem:injective-slwc}. By the composition
$(x,y)\mapsto x\mapsto(x,T(x))$, we have that $y=T(x)$.
\end{proof}
No-randomness-from-nothing follows just as before.
\begin{thm}
\label{thm:nrfn-CR-slwc}Assume $(2^{\mathbb{N}},\mu)$ is a computable
probability space, $T\colon(2^{\mathbb{N}},\mu)\rightarrow2^{\mathbb{N}}$
a Schnorr layerwise computable map, and $y\in2^{\mathbb{N}}$ a $\mu_{T}$-computably
random. Then $y=T(x)$ for some $\mu$-computable random $x\in2^{\mathbb{N}}$.\end{thm}
\begin{proof}
Let $\mu$ be a computable probability measure on $2^{\mathbb{N}}$,
let $T\colon(2^{\mathbb{N}},\mu)\rightarrow2^{\mathbb{N}}$ be a Schnorr
layerwise computable map, and let $y\in2^{\mathbb{N}}$ be $\mu_{T}$-computably
random. We want to find some $\mu$-computably random $x\in2^{\mathbb{N}}$
such that $y=T(x)$.

Let $\nu=\mu_{(\mathrm{id},T)}$, and let $x$ be blind $\nu(\cdot|y)$-Martin-Löf
random relative to $y$ (there are $\nu(\cdot|y)$-measure-one many,
so there is at least one). By Lemma~\ref{lem:vL}, $(x,y)$ is $\mu_{(\mathrm{id},T)}$-computably
random. By Lemma~\ref{lem:graph-slwc}, $x$ is computably random
and $y=T(x)$.
\end{proof}

\subsubsection{Isomorphism theorem for computable randomness and Schnorr layerwise
computable maps}

We end this subsection, by using Theorem~\ref{thm:nrfn-CR-slwc}
to give a Schnorr layerwise computable version of Lemma~\ref{lem:isomorphism}.
\begin{thm}
\label{thm:slwc-isomorphism}Let $\mu$ and $\nu$ be computable probability
measures on $2^{\mathbb{N}}$. Let $F\colon(2^{\mathbb{N}},\mu)\rightarrow(2^{\mathbb{N}},\nu)$
and $G\colon(2^{\mathbb{N}},\nu)\rightarrow(2^{\mathbb{N}},\mu)$
be Schnorr layerwise computable measure-preserving maps such that
\[
G(F(x))=x\ \ \mu\text{-a.e.}\qquad\text{and}\qquad F(G(y))=y\ \ \nu\text{-a.e.}
\]
Then $F$ and $G$ both conserve computable randomness, and if $x$
is $\mu$-computably random and $y$ is $\nu$-computably random then
\[
G(F(x))=x\qquad\text{and}\qquad F(G(y))=y.
\]
\end{thm}
\begin{proof}
Let $x$ be $\mu$-computably random. Then by Lemma~\ref{thm:nrfn-CR-slwc},
there is some $\nu$-computably random $y$ such that $G(y)=x$. It
remains to prove that $F(x)=y$. 

Assume not. Then $F(G(y))=F(x)\neq y$. Since $x$ and $y$ is are
computably random and since $F$ and $G$ are Schnorr layerwise computable
(with layerings $C_{n}$ and $D_{n}$ respectively), there is a large
enough $n$ such that $x\in C_{n}$ and $y\in D_{n}$. Then consider
the set 
\[
E_{n}=\{y_{0}\in D_{n}\cap G_{n}^{-1}(C_{n}):F_{n}(G_{n}(y_{0}))\neq y_{0}\}.
\]
This set is the intersection of the $\Pi_{1}^{0}$ set $D_{n}\cap G_{n}^{-1}(C_{n})$
and the $\Sigma_{1}^{0}$ set 
\[
\{y_{0}:y_{0}\notin D_{n}\cap G_{n}^{-1}(C_{n})\ \text{or}\ F_{n}(G_{n}(y_{0}))\neq y_{0}\}.
\]
So $E_{n}$ is $\Sigma_{2}^{0}$. Moreover, since $F$ and $G$ are
almost-surely inverses, $E_{n}$ has $\nu$-measure zero. Hence it
contains no $\nu$-Kurtz randoms, and therefore no $\nu$-computable
randoms. Since $y$ is $\nu$-computably random and in $E_{n}$, we
have a contradiction.

The other direction follows by symmetry.
\end{proof}

\subsection{\label{sub:comp-metric-spaces}Extending to computable metric spaces}

Theorem~\ref{thm:CR-nrfn} also applies to any computable metric
space as we will show in this section. Randomness on computable metric
spaces is well-understood; see \cite{Bienvenu2011,Hoyrup2009,Gacs2011,Rute:2013pd,rute1,Gacs:aa}.
In particular, Rute \cite{rute1} developed a theory of computable
randomness on a computable probability measure $(\mathbb{X},\mu)$
where $\mathbb{X}$ is a computable metric space and $\mu$ is a computable
measure on that space. Indeed the definition of computable randomness
(Definition~\ref{def:randomness}) can be extended to any computable
probability space by replacing (\ref{eq:test-pair1}) with the following
(Rute \cite[Thm~5.8]{rute1}), 
\[
\int_{B}t\,d\mu\leq\nu(B)\qquad(B\subseteq\mathbb{X}\ \text{Borel}).
\]

Moreover, each computable probability space $(\mathbb{X},\mu)$ is
isomorphic to a computable measure $\mu'$ on $2^{\mathbb{N}}$, via
a pair of almost-everywhere computable measure-preserving maps $I_{\mu}$
and $I_{\mu}^{-1}$ which commute (up to a.e.~equivalence) \cite[Prop.~7.9]{rute1}.
\[
\includegraphics{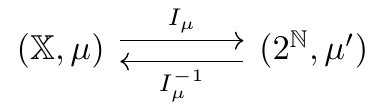}
\] Then any Schnorr layerwise computable measure-preserving map $T\colon(\mathbb{X},\mu)\rightarrow(\mathbb{Y},\nu)$
can be transferred to a Schnorr layerwise computable measure-preserving
map $T'\colon(2^{\mathbb{N}},\mu')\rightarrow(2^{\mathbb{N}},\nu')$
such that the following diagram commutes (up to a.e.~equivalence).
\[
\includegraphics{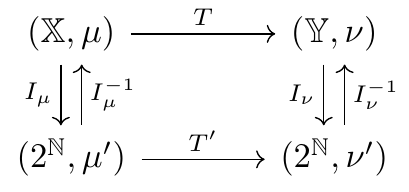}
\] That is, we set $T':=I_{\nu}\circ T\circ I_{\mu}^{-1}$. Then $T=I_{\nu}^{-1}\circ T'\circ I_{\mu}$
$\mu$-a.e. (Rute showed that a.e.~computable maps are Schnorr layerwise
computable maps \cite[p.~41, Prop.~3.24]{Rute:2013pd} and that Schnorr
layerwise computable maps are closed under composition \cite[p.~42, Prop.~3.27]{Rute:2013pd}.)

Now we can show the analogue of Theorem~\ref{thm:CR-nrfn}. Let $y\in\mathbb{Y}$
be $\nu$-random. By Theorem~\ref{thm:slwc-isomorphism} (also Rute
\cite[Prop~7.6, Thm.~7.11]{rute1}), we have that $y':=I_{\nu}(y)$
is $\nu'$-computably random and $I_{\nu}^{-1}(y')=I_{\nu}^{-1}(I_{\nu}(y))=y$.
Applying Theorem~\ref{thm:nrfn-CR-slwc}, there is some $x'$ which
is $\mu'$-computably random such that $T'(x')=y'$. Last, again using
Theorem~\ref{thm:slwc-isomorphism}, we have that $x:=I_{\mu}^{-1}(x')$
is $\mu$-computably random and $I_{\mu}(x)=x'$. Since $T$ and $I_{\nu}^{-1}\circ T'\circ I_{\mu}$
are equal $\mu$-almost-everywhere, then they are equal on all $\mu$-computable
randoms.\footnote{They are even equal on all Schnorr randoms \cite[p.~41, Prop.~3.18]{Rute:2013pd}.}
Therefore
\[
T(x)=I_{\nu}^{-1}(T'(I_{\mu}(x)))=y.
\]

\subsection{Generalizations to uniform relativization\label{sub:uniform-relativization}}

It is well known that one can relativize most proofs in computability
theory. In particular, one can straightforwardly relativize the proof
of Theorem~\ref{thm:CR-nrfn} to get the result that if $\mu$ is
computable from $a\in\mathbb{N}^{\mathbb{N}}$, $T\colon(2^{\mathbb{N}},\mu)\rightarrow2^{\mathbb{N}}$
is computable from $a$, and $y$ is $\mu_{T}$-computably random
relative to $a$, then there exists some $x\in\mathbb{N}^{\mathbb{N}}$
which is $\mu$-computably random relative to $a$ such that $T(x)=y$. 

The problem with this result is that this is not the best form of
relativization for computable randomness. It has been shown recently
that Schnorr and computable randomness behave better under uniform
relativization (which is related to truth-table reducibility) \cite{Franklin.Stephan:2010,Miyabe2011,Miyabe.Rute:2013}.

The full generalization of Theorem~\ref{thm:CR-nrfn} to uniform
relativization would be as follows. In the remainder of this section
assume that $\mathbb{X}$ and $\mathbb{Y}$ are computable metric
spaces and that $\{\mu_{a}\}_{a\in\mathbb{N}^{\mathbb{N}}}$ is a
family of probability measures on $\mathbb{X}$ such that $\mu_{a}$
is uniformly computable in $a\in\mathbb{N}^{\mathbb{N}}$. Say that
$x\in\mathbb{X}$ is \emph{$\mu_{a_{0}}$-computably random uniformly relativized to $a_{0}\in\mathbb{N}^{\mathbb{N}}$}
if $t_{a_{0}}(x)<\infty$ for all families $\{\nu_{a},t_{a}\}_{a\in\mathbb{N}^{\mathbb{N}}}$,
where $\nu_{a}$ is a probability measure on $\mathbb{X}$ uniformly
computable in $a\in\mathbb{N}^{\mathbb{N}}$, and $t_{a}$ is uniformly
lower semicomputable in $a$, such that 
\[
\int_{A}t_{a}\,d\mu_{a}\leq\nu_{a}(A)\quad(A\subseteq\mathbb{X}\ \text{Borel}).
\]

\begin{conjecture}
Assume $\{T_{a}\}_{a\in\mathbb{N}^{\mathbb{N}}}$ is a family of layerwise
maps $T_{a}\colon(\mathbb{X},\mu_{a})\rightarrow\mathbb{Y}$ where
$T_{a}$ is Schnorr layerwise computable uniformly in $a$. Fix $a_{0}\in\mathbb{N}^{\mathbb{N}}$.
If $y\in\mathbb{Y}$ is $(\mu_{a_{0}})_{T_{a_{0}}}$-computably random
uniformly relativized to $a_{0}$, then $y=T_{a_{0}}(x)$ for some
$x\in\mathbb{X}$ which is $\mu_{a_{0}}$-computably random uniformly
relativized to $a_{0}$.
\end{conjecture}
To prove this conjecture one would need to check that each step in
the proof of Theorem~\ref{thm:nrfn-CR-slwc} (and its generalization
to computable metrics spaces and Schnorr layerwise computable maps)
can be done uniformly in $a$. We see no issue in doing this (for
example, we are not aware of any nonuniform case analyses in the proof).
However, verifying all the details is beyond the scope of this paper.

\section{\label{sec:Application-of-nrfn}An application of no-randomness-from-nothing
for computable randomness}

There exist many applications of no-randomness-from-nothing in the
literature for Martin-Löf randomness. Most of these are direct applications.
By Theorems~\ref{thm:CR-nrfn} and the generalizations in Section~\ref{sec:3-generalizations},
these results also apply to computable randomness. In this section,
we would like to give another less obvious application. (Recall that
computable randomness can be naturally extended to computable measures
which are not necessarily probability measures.)
\begin{thm}
\label{thm:CR-sum-of-measures}Let $\mu$ be a probability measure
on $2^{\mathbb{N}}$ which is a countable sum of measures $\mu=\sum_{n}\mu_{n}$
where $\mu_{n}$ is computable in $n$. If $x$ is $\mu$-computably
random, then $x$ is $\mu_{n}$-computably random for some $n$.\end{thm}
\begin{proof}
Let $\nu$ be the measure which is basically a disjoint union of the
$\mu_{n}$. Specifically, let $0^{n}$ denote a sequence of $n$ many
$0$s. Then let $\nu$ be the measure given by $\nu(0^{n}1\sigma)=\mu_{n}(\sigma)$
and $\nu(\{0^{\infty}\})=0$. This measure is computable and it is
easy to see that $0^{n}1x$ is $\nu$-computably random if and only
if $x$ is $\mu_{n}$-computably random. Then let $T\colonsubseteq2^{\mathbb{N}}\rightarrow2^{\mathbb{N}}$
be the partial computable map given by $T(0^{n}1x)=x$. This map is
measure-preserving of type $T\colon(2^{\mathbb{N}},\nu)\rightarrow(2^{\mathbb{N}},\mu)$.
By no-randomness-from-nothing, each $\mu$-computable random $x$
comes from some $\nu$-computable random $0^{n}1x$. Therefore, $x$
is $\mu_{n}$-computably random for some $n$.
\end{proof}
The proof of Theorem~\ref{thm:CR-sum-of-measures} holds for any
randomness notion satisfying no-randomness-from-nothing and satisfying
the property that if $\nu(0^{n}1\sigma)=\mu_{n}(\sigma)$, $\nu(\{0^{\infty}\})=0$,
and $0^{n}1x$ is random for $\nu$ then $x$ is random for $\mu_{n}$.
In particular, Theorem~\ref{thm:CR-sum-of-measures} holds for Martin-Löf
randomness. 
\begin{question}
Does Theorem~\ref{thm:CR-sum-of-measures} hold for Schnorr randomness?\end{question}
\begin{rem}
The converse to Theorem~\ref{thm:CR-sum-of-measures} is as follows.
If $\mu\leq\nu$ are computable measures and $x$ is $\mu$-random,
then $x$ is $\nu$ random. For computable randomness, this is Lemma~\ref{lem:CR-smaller-to-bigger}.
It is also trivial to prove for many other randomness notions, including
Martin-Löf and Schnorr randomness. However, notice that, just as Theorem~\ref{thm:CR-sum-of-measures}
follows from no-randomness-from-nothing, its converse follows from
randomness conservation. The measure $\nu$ is the sum of the measures
$\mu$ and $\nu-\mu$. Take a disjoint sum of $\mu$ and $\nu-\mu$
(as in the proof of Theorem~\ref{thm:CR-sum-of-measures}), map this
disjoint sum to $\nu$, and apply randomness conservation.
\end{rem}

\section{\label{sec:NRFN-SR}No-randomness-from-nothing for Schnorr randomness}

In this section, we will prove the following theorem.
\begin{thm}
\label{thm:not-CR-nrnf}Let $\mu$ be a computable probability measure,
and assume $x_{0}\in2^{\mathbb{N}}$ is not computably random with
respect to $\mu$. Then there exists an almost-everywhere computable
measure-preserving map $T\colon(2^{\mathbb{N}},\lambda)\rightarrow(2^{\mathbb{N}},\mu)$
such that $T^{-1}(\{x_{0}\})=\varnothing$.
\end{thm}
As an obvious corollary we have the following.
\begin{cor}
\label{cor:SR-nrfn}Schnorr randomness does not satisfy no-randomness-from-nothing.
\end{cor}
The proof of Theorem~\ref{thm:not-CR-nrnf} fills up the rest of
this section. 

Fix $\mu$, and assume $x_{0}$ is not $\mu$-computably random. By
Lemma~\ref{lem:CR-standard-def} there is a computable measure $\nu'$,
such that $\nu'(x_{0}{\upharpoonright_{n}})/\mu(x_{0}{\upharpoonright_{n}})\rightarrow\infty$
as $n\rightarrow\infty$. For technical reasons we replace $\nu'$
with $\nu=\frac{1}{2}\mu+\frac{1}{4}\lambda+\frac{1}{4}\nu'$. We
still have that $\nu(x_{0}{\upharpoonright_{n}})/\mu(x_{0}{\upharpoonright_{n}})\rightarrow\infty$
as $n\rightarrow\infty$, but now $\nu(\sigma)>0$ for all $\sigma\in2^{*}$. 

Also, $\mu$ is absolutely continuous with respect to $\nu$, that
is $\nu(A)=0$ implies $\mu(A)=0$. By the Radon-Nikodym theorem there
is a \emph{density function} (or \emph{Radon-Nikodym derivative})
$f\colon2^{\mathbb{N}}\rightarrow\mathbb{R}$ such that for all $\sigma\in2^{*}$,
$\mu(\sigma)=\int_{[\sigma]}f\,d\nu$. The density is given by the
limit $f(x)=\lim_{n}\mu(x{\upharpoonright_{n}})/\nu(x{\upharpoonright_{n}})$.
This limit converges $\nu$-almost-everywhere by Lemma~\ref{lem:doob}.
Also $0\leq f\leq2$. To see that $f$ is the density, first apply
the dominated convergence theorem (which is applicable since $0\leq f\leq2$)
to get 
\[
\int_{[\sigma]}f(x)\,d\nu(x)=\lim_{n}\int_{[\sigma]}\frac{\mu(x{\upharpoonright_{n}})}{\nu(x{\upharpoonright_{n}})}\,d\nu(x)=\mu(\sigma),
\]
where the second equality comes from
\[
\int_{[\sigma]}\frac{\mu(x{\upharpoonright_{n}})}{\nu(x{\upharpoonright_{n}})}\,d\nu(x)=\mu(\sigma)\qquad(\text{for }n\geq|\sigma|).
\]
Notice by our construction that $f(x_{0})=0$. Also, the area under
the curve has $\nu\otimes\mathcal{L}$-measure one. (Here $\mathcal{L}$
is the Lebesgue measure on $\mathbb{R}$.) See Figure~\ref{fig:density}.

\begin{figure}[H]
\begin{centering}
\begin{center}
\includegraphics{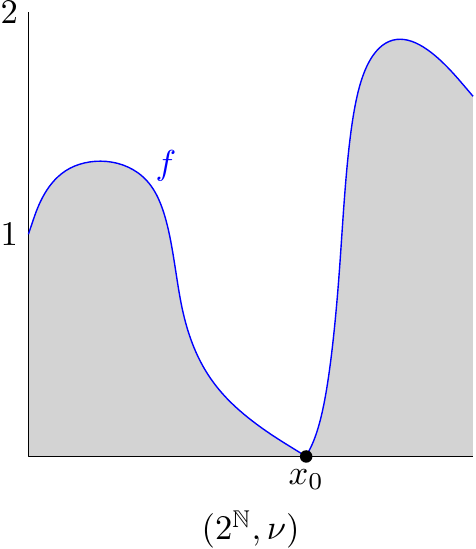}
\par\end{center}
\par\end{centering}

\caption{Graph of the density function $f$ of $\mu$ with respect to $\nu$.
The area under the curve has $\nu\otimes\mathcal{L}$-measure one.}
\label{fig:density}
\end{figure}

Our main tool is the following lemma. (We will be working on the computable
metric space $2^{\mathbb{N}}\times[0,2]$ for convenience. There is
no loss if we replace $[0,2]$ with $2^{\mathbb{N}}$, where $x\in[0,2]$
is identified with the binary expansion of $x/2$.) 
\begin{lem}
\label{lem:atomless-measures-maps} Assume $\mathbb{X}_{1}$ and $\mathbb{X}_{2}$
are computable metric spaces (e.g.~$2^{\mathbb{N}}$ and $2^{\mathbb{N}}\times[0,2]$).
If $\rho_{1}$ and $\rho_{2}$ are computable probability measures
on $\mathbb{X}_{1}$ and $\mathbb{X}_{2}$ respectively and $\rho_{1}$
is atomless, then one can construct an almost-everywhere computable
measure-preserving map $T\colon(\mathbb{X}_{1},\rho_{1})\rightarrow(\mathbb{X}_{2},\rho_{2})$
(uniformly in the codes for $\rho_{1}$ and $\rho_{2}$).\end{lem}
\begin{proof}
Since $\rho_{1}$ is atomless, $(\mathbb{X}_{1},\rho_{1})$ is isomorphic
to $(2^{\mathbb{N}},\lambda)$ \cite[Prop.~7.16]{rute1} (uniformly
in $\rho_{1}$). For any computable measure, there is an almost-everywhere
computable measure-preserving map $T\colon(2^{\mathbb{N}},\lambda)\rightarrow(\mathbb{X}_{2},\rho_{2})$
\cite[Prop.~7.9]{rute1}.
\end{proof}

\subsection{Description of the construction}

Consider the probability measure $\rho$ which is equal to $\nu\otimes\mathcal{L}$
restricted to the area under the curve of $f$. (See Figure~\ref{fig:density}.)
Unfortunately, this is not necessarily a computable measure on $2^{\mathbb{N}}\times[0,2]$. 

As motivation, let us consider for a moment the case where $\rho$
is computable and $f$ is also computable. Then it would be easy to
construct the desired $T$ as follows. Using Lemma~\ref{lem:atomless-measures-maps},
find a measure-preserving almost-everywhere computable map $T'\colon(2^{\mathbb{N}},\lambda)\rightarrow(2^{\mathbb{N}}\times[0,2],\rho)$.
Let $T\colon(2^{\mathbb{N}},\lambda)\rightarrow2^{\mathbb{N}}$ be
$T=\pi\circ T'$ where $\pi\colon2^{\mathbb{N}}\times[0,2]\rightarrow2^{\mathbb{N}}$
denotes the projection map $(x,y)\mapsto x$. The pushforward of $\lambda$
along $T$ is $\lambda_{T}=\mu$. Last, the preimage $\pi^{-1}(\{x_{0}\})$
is contained in $P=\{(x,y)\in2^{\mathbb{N}}\times[0,2]:f(x)=0\}$,
which is a $\rho$-measure-zero $\Pi_{1}^{0}$ subset of $2^{\mathbb{N}}\times[0,2]$
(since we are assuming $f$ is computable and $f(x_{0})=0$). If we
remove the set $P$ from the domain of $\pi$, we still have that
$T$ is almost-everywhere computable,\footnote{An almost-everywhere computable map is exactly a partial computable
map with a $\Pi_{2}^{0}$ domain of measure one. See Rute \cite[\S7]{rute1}
for more details.} but now $T^{-1}(\{x_{0}\})=\varnothing$.

In the general case where $\rho$ is not computable, we will make
this construction work by approximating the density function $f$
with a step function. Namely, let $f_{n}(x)=\mu(x{\upharpoonright_{n}})/\nu(x{\upharpoonright_{n}})$
and let 
\[
\mbox{\ensuremath{\rho}}_{n}=(\nu\otimes\mathcal{L})\upharpoonright\{(x,y)\in2^{\mathbb{N}}\times[0,2]:0<y<f_{n}(x)\}.
\]
Notice for each $n$ that $\rho_{n}$ is a probability measure since
$\int f_{n}\,d\nu=1$. We will build a sequence of a.e.~computable
functions $T_{n}$, where the $n$th function, $T_{n}$, transitions
from $\rho_{n-1}$ to $\rho_{n}$. While the limit of $T_{n}\circ\cdots\circ T_{0}$
may not be a.e.~computable, the first coordinate of its output will
be a.e.~computable, and that is enough to construct $\mu$ in the
desired manner.

\subsection{Requirements of the construction}

Formally, $T_{n}$ will satisfy the following requirements.
\begin{enumerate}
\item [($R_0$)]$T_{0}\colon(2^{\mathbb{N}},\lambda)\rightarrow(2^{\mathbb{N}}\times[0,2],\rho_{0})$
is an a.e.~computable measure-preserving map.
\item [($R_{n}$)]For all $n$,

\begin{enumerate}
\item $T_{n+1}\colon(2^{\mathbb{N}}\times[0,2],\rho_{n})\rightarrow(2^{\mathbb{N}}\times[0,2],\rho_{n+1})$
is measure-preserving and a.e.~computable in $n$.
\item If $(x',y')=T_{n+1}(x,y)$, then $x'{\upharpoonright_{n}}=x{\upharpoonright_{n}}$
and $0<y\leq y'\leq f_{n+1}(x')$.
\end{enumerate}
\item [($R_{\infty}$)]Let $\pi$ be the projection $\pi(x,y)=x$. Then
$T:=\lim_{n}\pi\circ T_{n}\circ\cdots\circ T_{0}$ satisfies

\begin{enumerate}
\item $T$ is well-defined, that is $\lim_{n}\pi\circ T_{n}\circ\cdots\circ T_{0}$
converges pointwise $\lambda$-almost-surely,
\item $T$ is $\lambda$-almost-everywhere computable,
\item $T$ is measure-preserving of type $(2^{\mathbb{N}},\lambda)\rightarrow(2^{\mathbb{N}},\mu)$
and
\item $T^{-1}(\{x_{0}\})=\varnothing$.
\end{enumerate}
\end{enumerate}

\subsection{Construction}

(See Figure~\ref{fig:stages} for an illustration of the first three
stages of the construction.)

\begin{figure}[h]
\begin{centering}
\subfloat[Stage 0]{\begin{centering}
\includegraphics[scale=0.8]{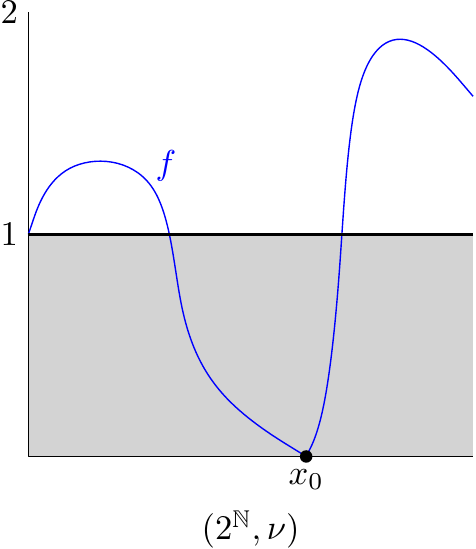}
\par\end{centering}

}\subfloat[Stage 1]{\begin{centering}
\includegraphics[scale=0.8]{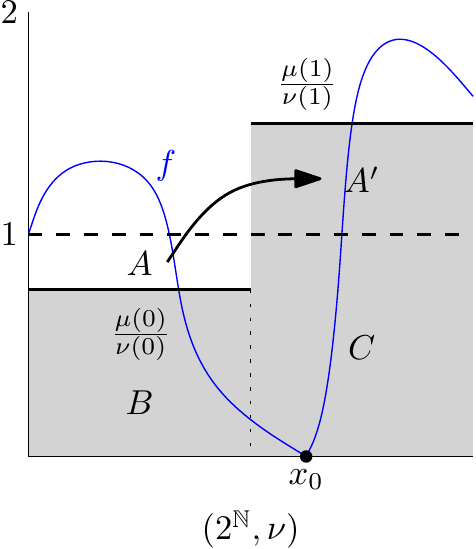}
\par\end{centering}

\label{subfig:Stage-1}}\subfloat[Stage 2]{\begin{centering}
\includegraphics[scale=0.8]{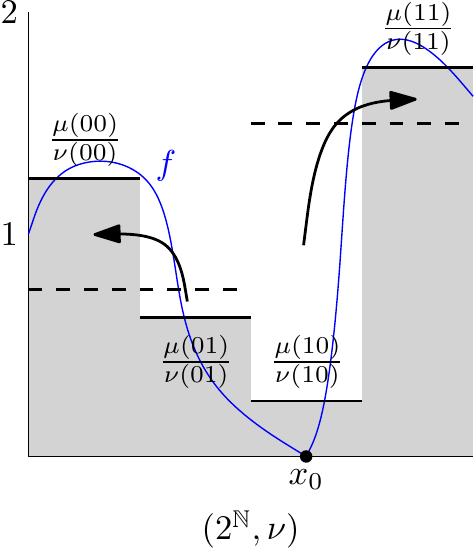}
\par\end{centering}

}
\par\end{centering}

\caption{First three stages of the construction. The solid line is $f_{n}$
and the dashed line is $f_{n-1}$.}
\label{fig:stages}
\end{figure}

\subsubsection*{Stage $0$.}

Our first approximation $f_{0}$ is just the constant function $\mathbf{1}$.
The measure $\rho_{0}$ mentioned above is supported on the rectangle
$2^{\mathbb{N}}\times[0,1]$. Using Lemma~\ref{lem:atomless-measures-maps}
one can compute an a.e.~computable measure-preserving map $T_{0}:(2^{\mathbb{N}},\lambda)\rightarrow(2^{\mathbb{N}}\times[0,2],\rho_{0})$.
Hence requirement ($R_{0}$) is satisfied.

\subsubsection*{Stage $1$.}

For illustration, let us handle Stage $1$ explicitly. The approximation
$f_{1}$, must satisfy $\int f_{1}\,d\nu=1$, therefore some of $f_{1}$
is below $f_{0}=\mathbf{1}$ and some of $f_{1}$ is above. As seen
in Subfigure~\ref{subfig:Stage-1}, the measure $\rho_{1}$ is the
same as $\rho_{0}$ except the mass of $\rho_{0}$ inside the rectangle
labeled $A$ has moved to the rectangle labeled $A'$. These two rectangles
must be of the same $\nu\otimes\mathcal{L}$ measure since $\int f_{1}\,d\nu=f_{0}$.
To construct $T_{1}$, divide up the support of $\rho_{0}$ into the
rectangles $A$, $B$, and $C$ as illustrated in Subfigure~\ref{subfig:Stage-1}.
For $(x,y)\in B,C$, set $T_{1}(x,y)=(x,y)$. 

For $A$, let $T_{1}\upharpoonright A$ be a $\nu\otimes\mathcal{L}$-measure-preserving
map from the interior of $A$ to the interior of $A'$. To construct
$T_{1}\upharpoonright A$ use Lemma~\ref{lem:atomless-measures-maps}
and the following facts: the measure $\nu\otimes\mathcal{L}$ is atomless,
the rectangles $A$ and $A'$ are computable, and the measures $\nu\otimes\mathcal{L}(A)$
and $\nu\otimes\mathcal{L}(A')$ are equal and computable. 

The remainder of $2^{\mathbb{N}}\times[0,2]$ (that is the set of
points not included in the interiors of $A$, $B$, or $C$) forms
an effectively closed set of $\rho_{0}$-measure zero, and therefore
we can leave it out of the domain of $T_{1}$. (In particular, $(x,0)$
is not in the domain of $T_{1}$ for any $x$.)

\subsubsection*{Stage $n+1$.}

This is the same idea as Stage $1$. For every $x$, we have $\int_{[x{\upharpoonright_{n}}]}f_{n+1}\,d\nu=f_{n}(x)$.
Therefore, we know that $\rho_{n+1}$ is the same as $\rho_{n}$ except
that for every $\sigma\in2^{n}$, some of the $\rho_{n}$-mass in
the column $[\sigma]\times[0,2]$ moves between the columns $[\sigma0]\times[0,2]$
and $[\sigma1]\times[0,2]$. If it moves, it must move upward, from
below $f_{n}$ to above $f_{n}$. Yet, none of the mass leaves the
column $[\sigma]\times[0,2]$. See Figure~\ref{fig:stages}.

Construct $T_{n+1}$ in the analogous way to $T_{1}$. In particular,
$T_{n+1}$ does not move mass out of any columns $[\sigma]\times[0,2]$
such that $|\sigma|=n$. It is straightforward to verify that all
the requirements of ($R_{n}$) are satisfied.

\subsection{Verification}

We already verified requirements ($R_{0}$) and ($R_{n}$). Requirement
($R_{\infty}$)(a), the a.e.~convergence of $T:=\lim_{n}\pi\circ T_{n}\circ\cdots\circ T_{0}$
is easy to verify since $T_{n+1}(x,y)$ does not change the first
$n$ bits of $x$. The same is true for the a.e.~computability of
$T$, that is requirement ($R_{\infty}$)(b).

To see that $T\colon(2^{\mathbb{N}},\lambda)\rightarrow(2^{\mathbb{N}},\mu)$
is measure-preserving (requirement ($R_{\infty}$)(c)) choose $\sigma\in2^{*}$.
Since, for all $n$, $T_{n+1}(x,y)$ does not change the first $n$
bits of $x$, $\omega\in T^{-1}[\sigma]$ if and only if both $\omega\in\operatorname{dom}T$
and $(T_{|\sigma|}\circ\cdots\circ T_{0})(\omega)\in[\sigma]\times[0,2]$.
(That is, after stage $n=|\sigma|$, the mass will not move between
columns of that size anymore.) Therefore, for $n=|\sigma|$, 
\begin{align*}
\lambda_{T}(\sigma) & =\lambda(T^{-1}[\sigma])\\
 & =\lambda((T_{n}\circ\cdots\circ T_{0})^{-1}([\sigma]\times[0,2]))\\
 & =\lambda_{T_{n}\circ\cdots\circ T_{0}}([\sigma]\times[0,2])\\
 & =\rho_{n}([\sigma]\times[0,2])\\
 & =\nu\otimes\mathcal{L}\left\{ (x,y)\in[\sigma]\times[0,2]:0<y<f_{n}(x)\right\} \\
 & =\int_{[\sigma]}f_{n}\,d\nu=\frac{\mu(\sigma)}{\nu(\sigma)}\nu(\sigma)=\mu(\sigma).
\end{align*}

Finally, for requirement ($R_{\infty}$)(d), assume for a contradiction
that there is some $\omega_{0}$ such that $T(\omega_{0})=x_{0}$.
Let $(x^{(n)},y^{(n)})=(T_{n}\circ\cdots\circ T_{0})(\omega_{0})$.
By all the parts of requirement ($R_{n}$)(b), we have that $y^{(n)}$
is nonzero, nondecreasing, and bounded above by 
\[
f_{n}(x^{(n)})=\frac{\mu(x^{(n)}{\upharpoonright_{n}})}{\nu(x^{(n)}{\upharpoonright_{n}})}=\frac{\mu(x_{0}{\upharpoonright_{n}})}{\nu(x_{0}{\upharpoonright_{n}})}.
\]
However, since $f(x_{0})=0$, this upper bound converges to $0$,
which contradicts that $y^{(n)}$ is nonzero and nondecreasing.

Therefore, we have proved that there is an a.e.~computable measure-preserving
map $T\colon(2^{\mathbb{N}},\lambda)\rightarrow(2^{\mathbb{N}},\mu)$
such that $T^{-1}(\{x_{0}\})=\varnothing$, proving Theorem~\ref{thm:not-CR-nrnf}.

\section{\label{sec:other-rand-notions}Randomness conservation and non-randomness-from-nothing
for difference randomness and 2-randomness}

In this section, we fill in the remaining rows of the table in the
introduction. While these results are not due to us, they are included
for completeness.

The $\lambda$-difference randoms can be characterized as the $\lambda$-Martin-Löf
randoms which do not compute $\emptyset'$ \cite{Franklin.Ng:2011}.
For simplicity, we will take this to be our definition of \emph{$\mu$-difference random}
for any computable measure $\mu$, as well.
\begin{prop}[{Bienvenu {[}personal comm.{]}}]
\label{prop:Bienvenu}Let $R$ be the class of Martin-Löf random
sequences which do not compute any element in $C\subseteq2^{\mathbb{N}}$
where $C$ is countable, then $R$ satisfies randomness conservation
and no-randomness-from-nothing.\end{prop}
\begin{proof}
Randomness conservation is easy. If $x\in R$ and $y=T(x)$ for a
measure-preserving a.e.~computable map $T\colon(2^{\mathbb{N}},\mu)\rightarrow(2^{\mathbb{N}},\nu)$,
then $y$ is Martin-Löf random (by randomness conservation for Martin-Löf
randomness), but $y$ cannot compute an element of $C$ (since $x$
cannot). Hence $y\in R$.

As for no-randomness-from-nothing, let $K_{n}$ be the complement
of the $n$th level of the universal Martin-Löf test and let $T\colon(2^{\mathbb{N}},\mu)\rightarrow(2^{\mathbb{N}},\nu)$
be a measure-preserving a.e.~computable map. Then $K_{n}\subseteq\operatorname{dom}T$
and $K_{n}$ is $\Pi_{1}^{0}$ in $n$. Fix a Martin-Löf random $y$.
Therefore $P_{n}^{y}:=\{x\in K_{n}:T(x)=y$\} is $\Pi_{1}^{0}$ in
$n$ and $y$. By no-randomness-from-nothing for Martin-Löf randomness
relativized to $y$, $P_{n}^{y}$ is nonempty for some $n$. By a
version of the cone avoidance theorem \cite[Thm 2.5]{Jockusch.Soare:1972}
relativized to $y$, there is a member of $P_{n}^{y}$ which does
not compute a member of $C$.
\end{proof}
Next, \emph{$2$-randomness} is Martin-Löf randomness relative to
$\emptyset'$.
\begin{prop}
\label{prop:2-rand-rp-nrfn}$2$-randomness satisfies randomness conservation
and no-randomness-from-nothing.\end{prop}
\begin{proof}
The proof for Martin-Löf randomness \cite[Thm.~3.2, 3.5]{BienvenuSubmitted}
can be relativized to $\emptyset'$.
\end{proof}

\section{\label{sec:Characterizing-Martin-L=0000F6f-random}Characterizing
Martin-Löf randomness via randomness conservation and no-randomness-from-nothing}

The concepts of no-randomness-from-nothing and randomness conservation
can be used to characterize various randomness notions. For example,
Theorems~\ref{thm:CR-nrfn} and \ref{thm:not-CR-nrnf}, together,
characterize computable randomness as the weakest randomness notion
satisfying no-randomness-from-nothing. 

Our main result in this section is a new characterization of Martin-Löf
randomness, but first let us consider a few other examples. Schnorr
characterized Schnorr randomness via randomness conservation and the
strong law of large numbers. 
\begin{prop}[{Schnorr \cite[Thm.~12.1]{Schnorr:1971rw}}]
\label{prop:Schnorr-SLLN}For $x\in2^{\mathbb{N}}$ the following
are equivalent.
\begin{enumerate}
\item $x$ is $\lambda$-Schnorr random.
\item For every almost-everywhere computable, measure-preserving map $F:(2^{\mathbb{N}},\lambda)\rightarrow(2^{\mathbb{N}},\lambda)$,
$F(x)$ satisfies the strong law of large numbers, that is 
\[
\frac{1}{n}\sum_{k=0}^{n-1}y(k)=1/2\quad\text{where }y=F(x).
\]

\end{enumerate}
\end{prop}
Notice (1) implies (2) follows from randomness conservation since
Schnorr randoms satisfy the strong law of large numbers. Gács, Hoyrup,
and Rojas \cite[Prop.~6]{Gacs2011} strengthened this result by showing
that ``measure-preserving map'' can be replaced with ``isomorphism.''

In Lemma~\ref{lem:CR-standard-def}, we characterized computable
randomness via the ratio $\nu(x{\upharpoonright_{n}})/\mu(x{\upharpoonright_{n}})$.
This function $\sigma\mapsto\nu(\sigma)/\mu(\sigma)$ is known as
a martingale and represents a gambling strategy. (See, for example,
\cite[\S6.3.1, Ch.~7]{Downey2010}\cite[Ch.~7]{Nies2009}.) Kolmogorov-Loveland
randomness (on the fair-coin measure $\lambda$) is a randomness notion
similar to computable randomness except that one can bet on the bits
out of order. Formally, call $F\colon2^{\mathbb{N}}\rightarrow2^{\mathbb{N}}$
a \emph{total computable nonmonotonic selection map} if $y=F(x)$
is computed as follows.
\begin{itemize}
\item At stage 0, the algorithm computes an index $i_{0}\in\mathbb{N}$
independently of $x$ and sets $y_{0}=x(i_{0})$.
\item At stage 1, the algorithm computes an index $i_{1}=i_{1}(y_{0})\in\mathbb{N}\smallsetminus\{i_{0}\}$
depending only on $y_{0}$ and sets $y_{1}=x(i_{1})$.
\item At stage $s+1$, we have calculated $y{\upharpoonright_{s+1}}$ where
$y_{n}=x_{i_{n}}$ for $0\leq n\leq s$. The algorithm chooses a new
index $i_{s+1}(y{\upharpoonright_{s+1}})\in\mathbb{N}\smallsetminus\{i_{0},i_{1},\ldots,i_{n}\}$
depending only on $y{\upharpoonright_{s+1}}$ and sets $y_{s+1}=x(i_{s+1})$.
\item The algorithm is total in that for all $x$ and $n$, the bit $(F(x))(n)$
is calculated.\footnote{Kolmogorov-Loveland randomness can be defined via non-total selection
maps as well \cite[\S3]{Merkle2003}.}
\end{itemize}
One can easily see that total computable nonmonotonic selection maps
are measure-preserving maps of type $(2^{\mathbb{N}},\lambda)\rightarrow(2^{\mathbb{N}},\lambda)$.

A sequence $x\in2^{\mathbb{N}}$ is \emph{Kolmogorov-Loveland random}
if given a total nonmonotonic selection map $F$, no computable $\lambda$-martingale
succeeds on $F(x)$ --- that is there is no computable measure $\nu$
such that $\liminf_{n}\nu(F(x{\upharpoonright_{n}}))/\lambda(F(x{\upharpoonright_{n}}))<\infty$.
Kolmogorov-Loveland randomness is between Martin-Löf randomness and
computable randomness, and it is a major open question whether Kolmogorov-Loveland
randomness equals Martin-Löf randomness. (See \cite[\S7.5]{Downey2010}\cite[\S7.6]{Nies2009}
for more about Kolmogorov-Loveland randomness.)

Bienvenu and Porter \cite[Thm.~4.2]{BienvenuSubmitted} noticed the
definition of Kolmogorov-Loveland randomness can be restated as in
this next proposition. They noticed that this proposition, when combined
with Muchnik's result that Kolmogorov-Loveland randomness is strictly
stronger than computable randomness \cite{Muchnik:1998rt}, implies
that computable randomness does not satisfy randomness conservation.
\begin{prop}
\label{prop:KL-CR}For $x\in2^{\mathbb{N}}$ the following are equivalent.
\begin{enumerate}
\item $x$ is $\lambda$-Kolmogorov-Loveland random.
\item $F(x)$ is $\lambda$-computably random for every total computable
nonmonotonic selection map $F\colon2^{\mathbb{N}}\rightarrow2^{\mathbb{N}}$.
\end{enumerate}
\end{prop}
Similarly Rute \cite{rute1} (based on a ``martingale process''
characterization of Martin-Löf randomness by Merkle, Mihailovi\'{c},
and Slaman \cite[\S4]{Merkle2006}) showed that Martin-Löf randomness
and computable randomness can be characterized in terms of each other.
\begin{prop}[{Rute \cite[Cor.~9.6]{rute1}}]
\label{prop:ML-CR-Rute}For $x\in2^{\mathbb{N}}$ the following are
equivalent.
\begin{enumerate}
\item $x$ is $\lambda$-Martin-Löf random.
\item $F(x)$ is $\lambda_{F}$-computably random for every almost-everywhere
computable map $F\colon(2^{\mathbb{N}},\lambda)\rightarrow2^{\mathbb{N}}$.
\end{enumerate}
\end{prop}
It is essential to Rute's proof that $\lambda_{F}$ may be not be
$\lambda$.

Combining Theorem~\ref{thm:not-CR-nrnf}, Proposition~\ref{prop:ML-CR-Rute},
and randomness conservation for Martin-Löf randomness we have the
following. (Here $\mathcal{M}_{\textnormal{comp}}^{1}$ denotes the
set of computable probability measures on $2^{\mathbb{N}}$.)
\begin{thm}
\label{thm:ML-characterization-measure}Let $A\subseteq2^{\mathbb{N}}\times\mathcal{M}_{\textnormal{comp}}^{1}$.
The following are equivalent.
\begin{enumerate}
\item $(x,\mu)$ is in $A$ if and only if $x$ is $\mu$-Martin-Löf random.
\item $A$ is the largest subset of $2^{\mathbb{N}}\times\mathcal{M}_{\textnormal{comp}}^{1}$
closed under no-randomness-from-nothing and randomness conservation.
\end{enumerate}
\end{thm}
\begin{proof}
The Martin-Löf randoms satisfy no-randomness-from-nothing and randomness
conservation.

Conversely, assume $A$ is strictly larger than the Martin-Löf randoms.
Then there is some $(y,\mu)\in A$ such that $y$ is not $\mu$-Martin-Löf
random. Now consider some almost-everywhere computable $F\colon(2^{\mathbb{N}},\lambda)\rightarrow(2^{\mathbb{N}},\mu)$,
which exists by Lemma~\ref{lem:atomless-measures-maps}.

By no-randomness-from-nothing for $A$, there is some $x\in2^{\mathbb{N}}$
such that $(x,\lambda)\in A$ and $F(x)=y$. By randomness conservation
for Martin-Löf randomness, $x$ cannot be $\lambda$-Martin-Löf random. 

By Proposition~\ref{prop:ML-CR-Rute}, there is some almost-everywhere
computable map $G\colon(2^{\mathbb{N}},\lambda)\rightarrow2^{\mathbb{N}}$
such that $G(x)$ is not $\lambda_{G}$-computably random. However,
$(G(x),\lambda_{G})\in A$ by randomness conservation.

By Theorem~\ref{thm:not-CR-nrnf}, there is an almost-everywhere
computable map $H\colon(2^{\mathbb{N}},\lambda)\rightarrow(2^{\mathbb{N}},\lambda_{G})$
such that $H^{-1}(\{G(x)\})=\varnothing$, contradicting no-randomness-from-nothing
for $A$. 
\end{proof}
We can get a stronger result by using a recent result of Petrovi\'{c}
\cite{Petrovic:}. He considered a variant of Kolmogorov-Loveland
randomness, where instead of betting on bits of $x$, one bets on
whether or not $x$ is in some clopen set $C\subseteq2^{\mathbb{N}}$
(a set is clopen if and only if it is a finite union of basic open
sets). Specifically, one computes a sequence $(C_{n})_{n\in\mathbb{N}}$
of clopen sets such that $\lambda(C_{n})=1/2$ for all $n$ and the
collection $\mathcal{C}=\{C_{n}\}_{n\in\mathbb{N}}$ is mutually $\lambda$-independent
in the sense of probability theory, that is for all finite $\mathcal{A}\subset\mathcal{C}$,
$\lambda(\bigcap_{C\in\mathcal{A}}C)=\prod_{C\in\mathcal{A}}\lambda(C)$.
This induces a total computable map $F(x)$ such that $(F(x))_{n}=1$
if and only if $x\in C_{n}$. Using notation similar to Petrovi\'{c}'s,
call such a map a \emph{sequence-set map}. Since the collection $\{C_{n}\}_{n\in\mathbb{N}}$
is mutually independent, a sequence-set map is a measure-preserving
map of type $F\colon(2^{\mathbb{N}},\lambda)\rightarrow(2^{\mathbb{N}},\lambda)$.
Framed in this way, Petrovi\'{c} proved the following.
\begin{thm}[Petrovi\'{c} \cite{Petrovic:}]
\label{thm:Petrovic}If $x$ is not $\lambda$-Martin-Löf random,
then there is a total computable sequence-set function $F:2^{\mathbb{N}}\rightarrow2^{\mathbb{N}}$
such that $F(x)$ is not $\lambda$-computably random.
\end{thm}
As a corollary, we have a strengthening of Propositions~\ref{prop:KL-CR}
and \ref{prop:ML-CR-Rute}.
\begin{cor}
\label{cor:ML-CR-Petrovic}For $x\in2^{\mathbb{N}}$ the following
are equivalent.
\begin{enumerate}
\item $x$ is $\lambda$-Martin-Löf random.
\item $F(x)$ is $\lambda$-computably random for every total computable
measure-preserving map $F\colon(2^{\mathbb{N}},\lambda)\rightarrow(2^{\mathbb{N}},\lambda)$.
\end{enumerate}
\end{cor}
Now we have a stronger version of Theorem~\ref{thm:ML-characterization-measure}
by basically the same proof, replacing Proposition~\ref{prop:ML-CR-Rute}
with Corollary~\ref{cor:ML-CR-Petrovic}. 
\begin{thm}
\label{thm:ML-characterization-lebsgue}The set of $\lambda$-Martin-Löf
randoms is the largest subset of $2^{\mathbb{N}}$ closed under no-randomness-from-nothing
and randomness conservation for a.e.\ computable measure-preserving
maps $F\colon(2^{\mathbb{N}},\lambda)\rightarrow(2^{\mathbb{N}},\lambda)$.\end{thm}
\begin{proof}
Let $A\subseteq2^{\mathbb{N}}$ be closed under no-randomness-from-nothing
and randomness conservation. By Theorem~\ref{thm:not-CR-nrnf}, every
element of $A$ must be $\lambda$-computably random. Then by Corollary~\ref{cor:ML-CR-Petrovic},
every element of $A$ must be $\lambda$-Martin-Löf random. Therefore
the $\lambda$-Martin-Löf randoms are the largest such set.
\end{proof}
\bibliographystyle{plain}
\bibliography{rand_from_somewhere}

\end{document}